\documentclass[a4paper, 10pt]{article}

\usepackage[all]{xy}
\usepackage{graphicx}
\usepackage{amsmath,amsthm}
\usepackage{amsfonts,amssymb}
\usepackage{float}
\usepackage{fancyhdr}
\usepackage{lscape}
\usepackage{enumerate}
\usepackage{stmaryrd}

\usepackage{hyperref}
\hypersetup{urlcolor=blue,linkcolor=black,citecolor=black,colorlinks=true}

\newtheorem{theorem}{Theorem}
\newtheorem{proposition}[theorem]{Proposition}

\newtheorem{lemma}[theorem]{Lemma}

\newtheorem{defi}[theorem]{Definition}

\theoremstyle{definition}
\newtheorem{example}[theorem]{Example}
\newtheorem{remark}[theorem]{Remark}

\newcommand{\E}{\mathbb{E}}

\newcommand{\Q}{\mathbb{Q}}
\newcommand{\R}{\mathbb{R}}
\newcommand{\Pb}{\mathbb{P}}

\newcommand{\F}{\mathcal{F}}
\renewcommand{\S}{\mathcal{S}}
\newcommand{\W}{\mathcal{W}}
\renewcommand{\L}{\mathcal{L}}

\newcommand{\equi}{\mathop{\sim}\limits}
\def\={{\;\mathop{=}\limits^{\text{(law)}}\;}}

\begin{document}

\begin{center}
{\LARGE \textbf{Penalizing null recurrent diffusions}}\\
\vspace{.7cm}
{\large Christophe \textsc{Profeta}\footnote{Laboratoire d'Analyse et Probabilit\'es, Universit\'e d'\'Evry - Val d'Essonne, B\^atiment I.B.G.B.I., 3\`eme \'etage, 23 Bd. de France, 91037 EVRY CEDEX.\\
 \textbf{E-mail}: christophe.profeta@univ-evry.fr}
}\\
\vspace{.8cm}
\end{center}

\textbf{Abstract:	} 
We present some limit theorems for the normalized laws (with respect to functionals involving last passage times at a given level $a$ up to time $t$) of a large class of null recurrent diffusions.  Our results rely on hypotheses on the L\'evy measure of the diffusion inverse local time at 0. As a special case, we recover some of the penalization results obtained by Najnudel, Roynette and Yor in the (reflected) Brownian setting. \\

\textbf{Keywords:} Penalization, null recurrent diffusions, last passage times, inverse local time.

\section{Introduction}
\subsection{A few notation}
We consider a linear regular null recurrent diffusion $(X_t, t\geq0)$  taking values in $\R^+$, with 0 an instantaneously reflecting boundary and $+\infty$ a natural boundary. Let $\Pb_x$ and $\E_x$ denote, respectively, the probability measure and the expectation associated with $X$ when started from $x\geq0$. We assume that $X$ is defined on the canonical space $\Omega:=\mathcal{C}(\R_+\rightarrow\R_+)$ and we denote by $(\F_t, t\geq0)$ its natural filtration, with  $\F_\infty:= \bigvee\limits_{t\geq0}\F_t$.\\
We denote by $s$ its scale function, with the normalization $s(0)=0$, and by $m(dx)$ its speed measure, which is assumed to have no atoms. 
It is known that $(X_t, t\geq0)$ admits a transition density $q(t,x,y)$ with respect to $m$, which is 
 jointly continuous and symmetric in $x$ and $y$, that is: $q(t,x,y)=q(t,y,x)$. This allows us to define, for $\lambda>0$, the resolvent kernel of $X$ by:
 \begin{equation}
u_\lambda(x,y)=\int_0^\infty e^{-\lambda t}q(t,x,y) dt. 
\end{equation}
 We also introduce $(L_t^a, t\geq0)$ the local time of $X$ at $a$, with the normalization:
$$
L_t^a:= \lim_{\varepsilon \downarrow 0} \frac{1}{m([a, a+\varepsilon[)} \int_0^t 1_{[a,a+\varepsilon[}(X_s)ds
$$
and $(\tau_l^{(a)}, l\geq0)$ the right-continuous inverse of $(L^a_t, t\geq0)$:
$$\tau_l^{(a)}:=\inf\{t\geq0; L_t^a>l\}.$$
As is well-known, $(\tau_l^{(a)}, l\geq0)$ is a subordinator, and we denote by $\nu^{(a)}$ its L\'evy measure.\\
 To simplify the notation, we shall write in the sequel $\tau_l$ for $\tau_l^{(0)}$ and $\nu$ for $\nu^{(0)}$. We shall also denote sometimes by $\overline{\mu}(t)=\mu([t,+\infty[)$ the tail of the measure $\mu$.

\subsection{Motivations}

Our aim in this paper is to establish some penalization results involving null recurrent diffusions.
Let us start by giving a definition of penalization: 
\begin{defi}
Let $(\Gamma_t,t\geq0)$ be a measurable process taking positive values, and such that $0<\E_x[\Gamma_t]<\infty$ for any $t>0$ and every $x\geq 0$. We say that the process $(\Gamma_t, t\geq0)$ satisfies the penalization principle if there exists a probability measure $\Q^{(\Gamma)}_x$ defined on $(\Omega, \F_\infty)$ such that:
$$\forall s\geq0, \;\forall \Lambda_s\in\F_s,\qquad \lim\limits_{t\rightarrow+\infty}\frac{\E_x[1_{\Lambda_s} \Gamma_t]}{\E_x[\Gamma_t]}=\Q^{(\Gamma)}_x(\Lambda_s).$$
\end{defi}
This problem has been widely studied by Roynette, Vallois and Yor when $\Pb_x$ is the Wiener measure or the law of a Bessel process (see \cite{RVY} for a synthesis and further references). They showed in particular that Brownian motion may be penalized by a great number of functionals involving local times, supremums, additive functionals,  numbers of downcrossings on an interval...  
Most of these results were then unified by Najnudel, Roynette and Yor (see \cite{NRY}) in a general penalization theorem, whose proof relies on the construction of a remarkable measure $\W$. \\

Later on, Salminen and Vallois managed in \cite{SV} to extend the class of diffusions for which penalization results hold. They proved in particular that under the assumption that the (restriction of the) L\'{e}vy measure $\frac{1}{\nu([1,+\infty[)}\nu_{|[1,+\infty[}$ of the subordinator $(\tau_l, l\geq0)$ is subexponential, the penalization principle holds for the functional $(\Gamma_t=h(L_t^0), t\geq0)$ with $h$ a  non-negative and non-increasing function with compact support. \\
Let us recall that a probability measure $\mu$ is said to be subexponential ($\mu$ belongs to class $\S$) if, for every $t\geq0$, 
$$\lim_{t\rightarrow +\infty}\frac{\mu^{\ast2}([t,+\infty[)}{\mu([t, +\infty[)} =2,$$
where $\mu^{\ast2}$ denotes the convolution of $\mu$ with itself.
The main examples of subexponential distributions are given by measures having a regularly varying tail (see Chistyakov \cite{Chi} or Embrechts, Goldie and Veraverbek \cite{EGV}): 
$$\mu([t,+\infty[) \equi_{t\rightarrow+\infty} \frac{\eta(t)}{t^\beta}$$
where $\beta\geq0$ and $\eta$ is a slowly varying function. When $\beta\in]0,1[$, we shall say that such a measure belongs to class $\mathcal{R}$. Let us also remark that a subexponential measure always satisfies the following property:
$$\forall x\in \R,\quad \lim_{t\rightarrow +\infty} \frac{\mu([t+x,+\infty[)}{\mu([t,+\infty[)}=1.$$
The set of such measures shall be denoted by $\L$, hence:
$$\mathcal{R} \subset \S \subset \L.$$
Now, following Salminen and Vallois, one may reasonably wonder what kind of penalization results may be obtained for diffusions whose normalized L\'evy measure belongs to classes $\mathcal{R}$ or $\L$. This is the main purpose of this paper, i.e. we shall prove that the results of Najnudel, Roynette and Yor remain true for diffusions whose normalized  L\'evy measure belongs to $\mathcal{R}$, and we shall give an ``integrated version'' when it belongs to $\L$\footnote{In the remainder of the paper, we shall make a  slight abuse of the notation and say that the measure $\nu$ belongs to $\L$ or $\mathcal{R}$ instead of $\frac{1}{\nu([1,+\infty[)}\nu_{|[1,+\infty[}$ belongs to $\L$ or $\mathcal{R}$. This is of no importance since  the fact that a probability measure belongs to classes $\L$ or $\mathcal{R}$ only involves the behavior of its tail at $+\infty$.}.

\subsection{Statement of the main results}

Let $a\geq0$, $g_a^{(t)}:=\sup\{u\leq t; X_u=a\}$ and $(F_t, t\geq0)$ be a positive and predictable process such that
 $$\displaystyle0<\E_x\left[\int_0^{+\infty} F_u dL_u^a\right]<\infty.$$

\noindent
\begin{theorem}\label{theo:equi}  \ \\ \vspace*{-.4cm}
\begin{enumerate}
\item If $\nu$ belongs to class $\mathcal{L}$, then
$$\forall a\geq0, \qquad \int_0^t\nu^{(a)}([s,+\infty[)ds \equi_{t\rightarrow+\infty}\int_0^t\nu([s,+\infty[)ds$$
and
$$\E_x\left[ \int_0^t F_{g_a^{(s)}}ds\right]\;\equi_{t\rightarrow+\infty}\; \left( \E_x[F_0](s(x)-s(a))^+ +
\E_x\left[\int_0^{+\infty} F_u dL_u^a\right]\right)\int_0^t\nu([s,+\infty[)ds .$$
\item If $\nu$ belongs to class $\mathcal{R}$:
$$\forall a\geq0, \qquad \nu^{(a)}([t,+\infty[) \equi_{t\rightarrow+\infty}\nu([t,+\infty[)$$
 and if $F$ is decreasing: 
$$\E_x\left[F_{g_a^{(t)}}\right]  \;\equi_{t\rightarrow+\infty}\;\left( \E_x[F_0](s(x)-s(a))^+ +
\E_x\left[\int_0^{+\infty} F_u dL_u^a\right]\right)\nu([t,+\infty[)$$
\end{enumerate}
\end{theorem}
\noindent
\begin{remark}
Point 2. does not hold for every $\nu\in \L$. Indeed, otherwise, taking $a=0$ and $F_t=1_{\{L_t^0\leq \ell\}}$ with $\ell>0$, one would obtain:
$$\Pb_0(L_t^0\leq \ell)=\Pb_0(\tau_\ell>t)  \;\equi_{t\rightarrow+\infty}\; \ell \nu([t,+\infty[),$$
a relation which is known to hold if and only if $\nu\in \S$, see \cite{EGV} or \cite[p.164]{Sat}.
\end{remark}

\begin{remark}
If $(X_t, t\geq0)$ is a positively recurrent diffusion, then $\int_0^{+\infty} \nu([s,+\infty[)ds =m(\R^+)$ and the limit in Point 1. equals:
$$\lim_{t\rightarrow +\infty} \E_x\left[ \int_0^t F_{g_a^{(s)}}ds\right]= \E_x\left[ \int_0^{+\infty} F_{g_a^{(s)}}ds\right] =  \E_x[F_0] \E_x\left[T_a\right] + \E_x\left[\int_0^{+\infty} F_u dL_u^a\right]m(\R^+).$$
\end{remark}

\noindent
In the following penalization result, we shall choose the weighting functional $\Gamma$ according to $\nu$:
\begin{theorem}\label{theo:penal}
Assume that: 

\begin{enumerate}[$a)$]
\item either $\nu$ belongs to class $\mathcal{L}$, and $\displaystyle \Gamma_t= \int_0^t F_{g_a^{(s)}}ds$,
\item or $\nu$ belongs to class $\mathcal{R}$ and $\displaystyle \Gamma_t= F_{g_a^{(t)}}$ with $F$ decreasing.
\end{enumerate}
Then, the penalization principle is satisfied by the functional $(\Gamma_t, t\geq0)$, i.e.  there exists a probability measure $\Q_x^{(F)}$  on $(\Omega, \F_\infty)$, which is the same in both cases, such that, 
$$\forall s\geq0, \;\forall \Lambda_s\in\F_s,\quad \lim\limits_{t\rightarrow+\infty}\frac{\E_x\left[1_{\Lambda_s}\Gamma_t\right]}{\E_x\left[\Gamma_t\right]}=\Q_x^{(F)}(\Lambda_s).$$
Furthermore:
\begin{enumerate}
\item The  measure $\Q_x^{(F)}$ is weakly absolutely continuous with respect to $\Pb_x$:
$$\Q_{x|\F_t}^{(F)}= \frac{M_t(F_{g_a})}{\E_x[F_0](s(x)-s(a))^++\E_x\left[\int_0^{+\infty} F_u dL_u^a\right]}\cdot \Pb_{x|\F_t}$$
where the martingale  $(M_t(F_{g_a}), t\geq0)$ is given by:
$$M_t(F_{g_a})=F_{g_a^{(t)}}(s(X_t)-s(a))^+ + \E_x\left[\int_t^{+\infty} F_u dL_u^a|\F_t \right].$$
\item Define $g_a:=\sup\{s\geq0, \; X_s=a\}$. Then, under $\Q_x^{(F)}$:
\begin{enumerate}[i)]
\item $g_a$ is finite a.s., 
\item conditionally to $g_a$, the processes $(X_t, t\leq g_a)$ and $(X_{g_a+t},t\geq0)$ are independent,   
\item the process $(X_{g_a+u}, u\geq0)$ is transient, goes towards $+\infty$ and its law does not depend on the functional $F$.
\end{enumerate}
\end{enumerate}
\end{theorem}
\noindent
We shall give in Theorem \ref{theo:IQ} a precise description of $\Q_x^{(F)}$ through an integral representation.

\begin{remark}
The main example of diffusion satisfying Theorems \ref{theo:equi} and \ref{theo:penal} is of course the Bessel process with dimension $\delta\in]0,2[$ reflected at 0.  Indeed, setting $\beta=1-\frac{\delta}{2}\in]0,1[$, the tail of its L\'evy measure at 0 equals:
$$\nu([t,+\infty[)= \frac{2^{1-\beta}}{\Gamma\left(\beta\right)}  \frac{1}{t^{\beta}} $$
i.e. $\nu\in\mathcal{R}$.
\end{remark}

\begin{remark}
Let us also mention that this kind of results no longer holds for positively recurrent diffusions. Indeed, it is shown in \cite{Pro} that if $(X_t, t\geq0)$ is a recurrent diffusion reflected on an interval, then, under mild assumptions, the penalization principle is satisfied by the functional  $(\Gamma_t=e^{-\alpha L_t^0}, t\geq0)$ with $\alpha\in \R$, but unlike in Theorem \ref{theo:penal}, the penalized process so obtained remains a positively recurrent diffusion.
\end{remark}

\begin{example}
Assume that $\nu\in \mathcal{R}$ and let $h$ be a positive and decreasing function with compact support on $\R^+$.
\begin{enumerate}[$\bullet$]
\item Let us take $(F_t, t\geq0)=(h(L_t^a), t\geq0)$.\\
Then $\displaystyle\E_0\left[\int_0^{+\infty} h(L_s^a) dL_s^a\right]=\int_0^{+\infty}h(\ell)d\ell<\infty$ and, since $L_{g_a^{(t)}}^a=L_t^a$,
$$\E_0\left[h(L_t^a)\right]\;\equi_{t\rightarrow+\infty}\;\nu([t,+\infty[) \int_0^{+\infty} h(\ell)d\ell,$$
and the martingale $(M_t(L_{g_a}^a), t\geq0)$ is an Az\'ema-Yor type martingale:
$$M_t(L_{g_a}^a)=h(L_t^a)(s(X_t)-s(a))^+ +\int_{L_t^a}^{+\infty} h(\ell)d\ell.$$

\item Let us take $(F_t, t\geq0) = (h(t), t\geq0)$.\\
Then $\displaystyle\E_0\left[\int_0^{+\infty} h(u) dL_u^a\right]=\int_0^{+\infty}h(u)\E_0[dL_u^a]=\int_0^{+\infty} h(u)q(u,0,a)du<\infty$ and therefore:
$$\E_0\left[h(g_a^{(t)})\right]\equi_{t\rightarrow+\infty}\nu([t,+\infty[) \int_0^{+\infty} h(u)q(u,0,a)du,$$
and the martingale $(M_t(g_a), t\geq0)$ is given by:
$$M_t(g_a)=h(g_a^{(t)})(s(X_t)-s(a))^+ +\int_0^{+\infty} h(v+t)q(v,X_t,a)dv.$$

\item One may also take for instance $(F_t, t\geq0)=(h(S_t), t\geq0)$ where $S_t:=\sup\limits_{s\leq t}X_s$ or $(F_t, t\geq0)=h\left(\int_0^t f(X_s) ds\right)$ where $f:\R^+\longrightarrow \R^+$ is a Borel function. These were the first kind of weights studied by Roynette, Vallois and Yor, see \cite{RVY2} and \cite{RVY3}.
\end{enumerate}

\end{example}

\subsection{Organization}

The remainder of the paper is organized as follows:

\begin{enumerate}[$\bullet$]
\item In Section 2, we introduce some notation and recall a few known results that we shall use in the sequel. They are mainly taken from \cite{Sal} and \cite{SVY}.
\item Section 3 is devoted to the proof of Theorem \ref{theo:equi}. The two Points 1. and 2. are dealt with separately: when $\nu\in \mathcal{R}$, the asymptotic is obtained via a Laplace transform and a Tauberien theorem, while in the case $\nu\in \mathcal{L}$,  we shall use a basic result on integrated convolution products.
\item  Section 4 gives the proof of  Point 1. of Theorem \ref{theo:penal}, which essentially relies on a meta-theorem, see \cite{RVY}.
\item In Section 5, we derive a integral representation for the penalized measure $\Q_x^{(F)}$ which implies Point 2. of Theorem \ref{theo:penal}.
\item Finally, Section 6 is devoted to prove that, with our normalizations, the process $(N_t^{(a)}:=(s(X_t)-s(a))^+-L_t^a, t\geq0)$ is a martingale.
\end{enumerate}

\section{Preliminaries}

\noindent
In this section, we essentially recall some known results that we shall need in the sequel.\\

\noindent
$\bullet$ Let  $T_a:=\inf\{u\geq0; X_u=a\}$ be the first passage time of $X$ to level $a$. Its Laplace transform is given by
\begin{equation}\label{eq:Ta}\E_x\left[e^{-\lambda T_a}\right]=\frac{u_\lambda(a,x)}{u_\lambda(a,a)}.
\end{equation}
Since $(X_t, t\geq0)$ is assumed to be null recurrent, we have for $x>a$, $\E_x[T_a]=+\infty$.

\noindent
$\bullet$ We define $(\widehat{X}_t, t\geq0)$  the diffusion $(X_t, t\geq0)$ killed at $a$:
$$
\widehat{X}_t:= 
\begin{cases} 
X_t & \text{ $t<T_a$},\\
\partial & \text{ $t\geq T_a$}.
\end{cases}
$$
where $\partial$ is a cemetary point.
We denote by  $\widehat{q}(t,x,y)$  its transition density with respect to $m$:
$$\widehat{\Pb}_x(\widehat{X}_t\in dy)= \widehat{q}(t,x,y)m(dy)=\Pb_x\left(X_t\in dy; t<T_a\right).$$

\noindent
$\bullet$ We also introduce $(X^{\uparrow a}_t, t\geq0)$ the diffusion $(\widehat{X}_t, t\geq0)$ conditionned not to touch $a$, following the construction in \cite{SVY}. For $x>a$ and $F_t$ a positive, bounded and $\F_t$-measurable r.v.:
$$
\E_x^{\uparrow a}\left[F_t\right]=\frac{1}{s(x)-s(a)} \E_x\left[F_t (s(X_t)-s(a))1_{\{t<T_a\}} \right]. 
$$
By taking $F_t=f(X_t)$, we deduce in particular that, for $x,y>a$:
$$q^{\uparrow a}(t,x,y)=\frac{\widehat{q}(t,x,y)}{(s(x)-s(a))(s(y)-s(a))}\quad \text{ and }\quad m^{\uparrow a}(dy)= (s(y)-s(a))^2 m(dy).$$
Letting  $x$ tend towards $a$, we obtain:
\begin{equation*}
q^{\uparrow a}(t,a,y)=\frac{n_{y,a}(t)}{s(y)-s(a)} \quad \text{ where }\; \Pb_y(T_a\in dt)=: n_{y,a}(t)dt.\end{equation*}

\noindent
$\bullet$ We finally define $(X^{x,t,y}_u,u\leq t)$ the bridge of $X$ of length $t$ going from $x$ to $y$. Its law may be obtained as a \textit{$h$}-transform, for $u<t$:
\begin{equation}\label{W12 mart}
\E^{x,t,y}\left[F_u\right]=\E_x\left[\frac{q(t-u,X_u,y)}{q(t,x,y)}F_u\right].
\end{equation}

With these notation, we may state the two following Propositions which are essentially due to Salminen.
\begin{proposition}[\cite{Sal}]\label{prop:ga} \ \\ \vspace*{-.3cm}
\begin{enumerate}
\item The law of $g_a^{(t)}:=\sup\{u\leq t; X_u=a\}$ is given by: 
\begin{equation}\label{gat}
\Pb_x(g_a^{(t)}\in du)=\Pb_x(T_a>t) \delta_0(du) + q(u,x,a)\nu^{(a)}([t-u, +\infty[)du.
\end{equation}
\item On the event $\{X_t> a\}$, the density of the couple $(g_a^{(t)}, X_t)$ reads :
\begin{equation}\label{eq:gatX}
\Pb_x\left(g_a^{(t)}\in du, X_t\in dy\right)=\Pb_x(T_a>t, X_t\in dy)\delta_0(du)+ \frac{q(u,x,a)}{s(y)-s(a)}\Pb_a^{\uparrow a}(X_{t-u}\in dy) du \quad (y>a)
\end{equation}
\end{enumerate}
\end{proposition}
\noindent

\noindent
We now study the pre- and post- $g_a^{(t)}$-process:
\begin{proposition}\label{prop:+-ga} Under $\Pb_x$: 
\begin{enumerate}[$i)$]
\item Conditionnally to $g_a^{(t)}$, the process $(X_s, s\leq g_a^{(t)})$ and 
$(X_{g_a^{(t)}+s}, s\leq t-g_a^{(t)})$ are independent.
\item Conditionnally to $g_a^{(t)}=u$,
$$(X_s, s \leq u) \= (X_s^{x,u,a}, s\leq u).$$
\item Conditionnally to $g_a^{(t)}=u$ and $X_t=y>a$,
$$(X_{u+s}, s\leq t-u)\= \left(X_s^{\uparrow a\; a,t-u,y}, s\leq t-u \right).$$
\end{enumerate}
\end{proposition}
\noindent
\begin{proof}
$i)$ Point $(i)$ follows from Proposition 5.5 of \cite{Mil} applied to the diffusion 
$$\displaystyle X_s^{(t)}:= 
\begin{cases}
X_s & \text{ $s<t$}\\
\partial & \text{ $s\geq t$}
\end{cases}$$
so that $\xi:=\inf\{s\geq 0;\; X_s^{(t)}\notin \R^+ \}=t$.\\
\\
$ii)$ Point $(ii)$ is taken from \cite{Sal}.\\
$iii)$ As for Point $(iii)$, still from $\cite{Sal}$, conditionnally to $g_a^{(t)}=u$ and $X_t=y>a$, we have:
$$(X_{u+s}, s\leq t-u)\= \left(\widehat{X}_s^{a,t-u,y}, s\leq t-u \right).$$
But the bridges of $\widehat{X}$ et $X^\uparrow$ have the same law. Indeed, for $y,x>a$:\\
\\
$\displaystyle\widehat{\Pb}^{\,x,t,y}\left(X_{t_1}\in dx_1, \ldots, X_{t_n}\in dx_n\right)$
\vspace{-.2cm}
\begin{align*}
&= \widehat{\E}_x\left[\frac{\widehat{q}(t-t_n,X_{t_n},y)}{\widehat{q}(t,x,y)}1_{\{X_{t_1}\in dx_1, \ldots, X_{t_n}\in dx_n\}}   \right]\quad \text{(from (\ref{W12 mart}))}\\
&=\E_x\left[\frac{(s(X_{t_n})-s(a))q^{\uparrow a}(t-t_n,X_{t_n},y)}{(s(x)-s(a)) q^{\uparrow a}(t,x,y)}1_{\{X_{t_1}\in dx_1, \ldots, X_{t_n}\in dx_n\}}1_{\{t_n<T_a\}} \right]\\
&=\E_x^{\uparrow a}\left[\frac{q^{\uparrow a}(t-t_n,X_{t_n},y)}{q^{\uparrow a}(t,x,y)}1_{\{X_{t_1}\in dx_1, \ldots, X_{t_n}\in dx_n\}}\right]\quad \text{(by definition of $\Pb^{\uparrow a}_x$)}\\
&= \Pb^{\uparrow a\; x,t,y}\left(X_{t_1}\in dx_1, \ldots, X_{t_n}\in dx_n\right).
\end{align*}
and the result follows by letting $x$ tend toward $a$.\\
\end{proof}

\section{Study of asymptotics}

The aim of this section is to prove Theorem \ref{theo:equi}. We start with the case $\nu\in \mathcal{R}$.

\subsection{Proof of Theorem \ref{theo:equi} when $\nu\in \mathcal{R}$}

Let  $(F_t, t\geq0)$ be a decreasing, positive and predictable process such that
 $$\displaystyle0<\E_x\left[\int_0^{+\infty} F_u dL_u^a\right]<\infty.$$
 
Our approach in this section is based on the study of the Laplace transform of $t\longmapsto \E_x\left[F_{g_a^{(t)}}\right]$. Indeed, from Propositions \ref{prop:ga} and \ref{prop:+-ga}, we may write, applying Fubini's Theorem:
\begin{align}
\notag&\int_0^{+\infty}e^{-\lambda t}\E_x\left[F_{g_a^{(t)}}\right]dt\\
\notag&= \int_0^{+\infty}e^{-\lambda t}\int_0^t \E_x\left[F_u|g_a^{(t)}=u\right]\Pb(g_a^{(t)}\in du) dt\\
\notag&=\E_x[F_0] \int_0^{+\infty} e^{-\lambda t}  \Pb_x(T_a>t) dt +  \int_0^{+\infty} e^{-\lambda t} \int_0^t  \E_x\left[F_u|X_u=a\right]  q(u,x,a)\nu^{(a)}([t-u,+\infty[)du\,dt\\
\label{eq:LFga}&= \E_x\left[F_0\right] \frac{1-\E_x\left[e^{-\lambda T_a}\right]}{\lambda}
+\int_0^{+\infty} e^{-\lambda t} \Pb^{x,t,a}(F_t)  q(t,x,a) dt \times\int_0^{+\infty} e^{-\lambda t} \nu^{(a)}([t,+\infty[)dt
\end{align}
We shall now study the asymptotic (when $\lambda \rightarrow 0$) of each term separately.  To this end, we state and prove two Lemmas.

\subsubsection{The Laplace transform of $t\rightarrow\nu^{(a)}([t,+\infty[)$}
\begin{lemma}\label{lem:Lnu}
The following formula holds:
$$\frac{1}{\lambda u_\lambda(a,a)}=\int_0^{+\infty} e^{-\lambda t} \nu^{(a)}([t,+\infty[)dt$$
\end{lemma}

\begin{proof}
Since $\tau$ is a subordinator and $m$ has no atoms, from the L\'evy-Khintchine formula:
$$\E_a\left[e^{-\lambda \tau_l^{(a)}}\right]=\exp\left(l\int_0^{+\infty}  (1-e^{-\lambda t})\nu^{(a)}(dt)\right).$$
Then, from the classic relation:
$$\E_a\left[e^{-\lambda \tau_l^{(a)}}\right]=e^{-l/u_{\lambda(a,a)}}$$
we deduce that 
$$\frac{1}{u_\lambda(a,a)}=\int_0^{+\infty}  (1-e^{-\lambda t})\nu^{(a)}(dt).$$
Now, let $\varepsilon>0$ :
\begin{eqnarray*}
\int_\varepsilon^\infty (1-e^{-\lambda t})\nu^{(a)}(dt)&=&\big[(e^{-\lambda t}-1)\nu^{(a)}([t,+\infty[)\big]_\varepsilon^{+\infty} + \int_\varepsilon^\infty\lambda e^{-\lambda t} \nu^{(a)}([t,+\infty[) dt\\
&=&(1-e^{-\lambda \varepsilon}) \nu^{(a)}([\varepsilon,+\infty[) + \int_\varepsilon^\infty\lambda e^{-\lambda t} \nu^{(a)}([t,+\infty[ )dt
\end{eqnarray*}
Since both terms are positive, we may let $\varepsilon\rightarrow 0$ to obtain:
\begin{equation*}
\frac{1}{\lambda u_\lambda(a,a)} = \int_0^\infty e^{-\lambda t} \nu^{(a)}([t,+\infty[) dt + \ell,
\end{equation*}
where $\ell:=\lim\limits_{ \varepsilon\rightarrow 0}  \varepsilon  \nu([\varepsilon,+\infty[)$, and it remains to prove that $\ell=0$. Assume that $\ell>0$. Then:
$\displaystyle \nu^{(a)}([\varepsilon,+\infty[)\equi_{\varepsilon\rightarrow0}\frac{\ell}{\varepsilon}$ and :
\begin{align*}
\int_\varepsilon^1 t\nu^{(a)}(dt) &= \big[-t \nu^{(a)}([t,1]) \big]_\varepsilon^1 + \int_\varepsilon^1 \nu^{(a)}([t,1]) dt\\
&= \varepsilon\nu^{(a)}([\varepsilon,1])+\int_\varepsilon^1 \nu^{(a)}([t,1]) dt\\
&\xrightarrow[\varepsilon\rightarrow0]{} +\infty,
\end{align*}
since, from our hypothesis,  $\displaystyle\nu^{(a)}([t,1]) \equi_{u \rightarrow 0} \frac{\ell}{t}$, i.e. $t\mapsto\nu^{(a)}([t,1])$ is not integrable at 0. This contradicts the fact that $\nu^{(a)}$ is the L\'evy measure of a subordinator, hence $\ell=0$ and the proof is completed.\\
\end{proof}

\begin{remark}
Since we assume that $(X_t, t\geq0)$ is a null recurrent diffusion, we have $m(\R^+)=+\infty$ and  from Salminen \cite{Sal3}:
\begin{equation}\label{eq:Sal}
\lim_{\lambda\rightarrow 0} \lambda u_\lambda(a,a)=\frac{1}{m(\R^+)}=0.
\end{equation}
Thus, from the monotone convergence theorem, the function $t\rightarrow \nu^{(a)}([t,+\infty[)$ is not integrable at $+\infty$.
On the other hand, if $(X_t, t\geq0)$ is positively recurrent, we obtain:
$$\int_0^{+\infty}  \nu^{(a)}([t,+\infty[) dt =m(\R^+) <+\infty.$$
\end{remark}

\noindent
We now study the asymptotic of the first hitting time of $X$ to level $a$.
\begin{lemma}\label{lem:TaR}
Let $x>a$ and assume that $\nu$ belongs to class $\mathcal{R}$. Then:
\begin{enumerate}[$i)$]
\item  The tails of $\nu$ and $\nu^{(a)}$ are equivalent:
$$\displaystyle \nu^{(a)}([t,+\infty[) \equi_{t\rightarrow+\infty}\nu([t,+\infty[).$$
\item The survival function of $T_a$ satisfies the following property: 
\begin{equation}\label{eq:equi2}
 \Pb_x(T_a\geq t) \equi_{t\rightarrow+\infty}\; (s(x)-s(a)) \nu([t,+\infty[).
 \end{equation}

\end{enumerate}

\end{lemma}

\begin{proof}
We shall use the following Tauberian theorem (see Feller \cite[Chap. XIII.5, p.446]{Fel} or \cite[Section 1.7]{BGT}):
\\

\textit{Let $f$ be a positive and decreasing function, $\beta\in]0,1[$ and $\eta$ a slowly varying function. Then,}
\begin{equation}\label{eq:Taub}
f(t)\; \equi_{t\rightarrow+\infty} \;\frac{\eta(t)}{t^{\beta}} \quad\Longleftrightarrow\quad 
\int_0^\infty e^{-\lambda t}f(t)dt\; \equi_{\lambda\rightarrow0}\;\frac{\Gamma(\beta)}{\lambda^{1-\beta}}\eta\left(\frac{1}{\lambda}\right).
\end{equation}
In particular, with $f(t)=\nu([t,+\infty[)$ (since $\nu\in\mathcal{R}$), we obtain:
$$\int_0^\infty e^{-\lambda t} \nu([t, +\infty[) dt =\frac{1}{\lambda u_\lambda(0,0)} \; \equi_{\lambda\rightarrow0}\;\frac{\Gamma(\beta)}{\lambda^{1-\beta}}\eta\left(\frac{1}{\lambda}\right).$$
Now, from Krein's Spectral Theory (see for instance \cite[Chap.5]{DMK}, \cite{KK}, \cite{KW} or \cite{Kas}), $u_\lambda(x,y)$ admits the representation, for $x\leq y$:
\begin{equation}\label{eq:Krein}
u_\lambda(x,y)=\Phi(x,\lambda)\left(u_\lambda(0,0)\Phi(y,\lambda)-\Psi(y,\lambda)\right)
\end{equation}
where the eigenfunctions $\Phi$ and $\Psi$ are solutions of:
$$\begin{cases}
\displaystyle\Phi(x,\lambda)=1+\lambda \int_0^x s^\prime(dy) \int_0^y \Phi(z,\lambda)m(dz),\\
\displaystyle\Psi(x,\lambda)=s(x)+\lambda \int_0^x s^\prime(dy) \int_0^y \Psi(z,\lambda)m(dz),
\end{cases}$$
We deduce then, since $\displaystyle \lim_{\lambda\rightarrow 0}\Phi(x,\lambda)=1$, $\displaystyle\lim_{\lambda\rightarrow 0}\Psi(x,\lambda)=s(x)$ and  $\displaystyle\lim_{\lambda\rightarrow 0}u_\lambda(0,0)=+\infty$ that:
$$\frac{u_{\lambda}(a,a)}{u_\lambda(0,0)}=\Phi(a,\lambda)^2 - \frac{\Phi(a,\lambda)\Psi(a,\lambda)}{u_\lambda(0,0)}\xrightarrow[\lambda\rightarrow 0]{}1.$$
Therefore, from the Tauberien theorem (\ref{eq:Taub}) with $f(t)=\nu^{(a)}([t,+\infty[)$, we obtain:
$$\nu^{(a)}([t, +\infty[)\; \equi_{t\rightarrow+\infty} \;\frac{\eta(t)}{t^{\beta}}$$
i.e. Point $(i)$ of Lemma \ref{lem:TaR}.\\
\noindent
To prove Point $(ii)$, let us compute the Laplace transform of $\Pb_x(T_a\geq t)$, using (\ref{eq:Ta}):
\begin{equation}\label{eq:LTa}
\int_0^{+\infty} e^{-\lambda t}\Pb_x(T_a\geq t) dt =\frac{1-\E_x\left[e^{-\lambda T_a}\right]}{\lambda} =\frac{1}{\lambda}-\frac{u_\lambda(x,a)}{\lambda u_{\lambda}(a,a)}= \frac{u_\lambda(a,a)-u_\lambda(x,a)}{\lambda u_{\lambda}(a,a)}.
\end{equation}
\noindent
Now, for $x>a$, we get from (\ref{eq:Krein}): 
\begin{align*}
u_\lambda(a,a)-u_\lambda(a,x)&=\Phi(a,\lambda)(u_\lambda(0,0)\Phi(a,\lambda)-\Psi(a,\lambda))-
 \Phi(a,\lambda)(u_\lambda(0,0)\Phi(x,\lambda)-\Psi(x,\lambda))\\
 &= \Phi(a,\lambda)u_\lambda(0,0)\left(\Phi(a,\lambda)-\Phi(x,\lambda)\right)+\Phi(a,\lambda)\left(\Psi(x,\lambda)-\Psi(a,\lambda)\right)\\
 &= \Phi(a,\lambda)u_\lambda(0,0)\left(\lambda \int_a^x s^\prime(y)dy\int_0^y\Phi(z,\lambda)m(dz)\right)+\Phi(a,\lambda)\left(\Psi(x,\lambda)-\Psi(a,\lambda)\right),
\end{align*}
and, letting $\lambda$ tend toward 0 and using (\ref{eq:Sal}):
$$\lim_{\lambda\rightarrow 0}u_\lambda(a,a)-u_\lambda(a,x)=s(x)-s(a).$$
Therefore,
$$\int_0^{+\infty} e^{-\lambda t}\Pb_x(T_a\geq t) dt \;\equi_{\lambda\rightarrow 0}\; \frac{s(x)-s(a)}{\lambda u_\lambda(a,a)}\;\equi_{\lambda\rightarrow 0}\;(s(x)-s(a))   \frac{\Gamma(\beta)}{\lambda^{1-\beta}}\eta\left(\frac{1}{\lambda}\right) 
$$
and Point $(ii)$ follows once again from the Tauberian theorem (\ref{eq:Taub}). \\
\end{proof}

\subsubsection{Proof of Point 2. of Theorem \ref{theo:equi}}

We now let $\lambda$ tend toward 0 in (\ref{eq:LFga}). Observe first that, from our hypothesis on  $(F_u, u\geq0)$:
$$
 \int_0^{+\infty}\Pb^{x,u,a}(F_u)  q(u,x,a) du
= \int_0^{+\infty}\E_x\left[F_u|X_u=a\right]  \E_x[dL_u^a]  =\E_x\left[\int_0^{+\infty}F_udL_u^a\right] <+\infty.
$$
Then, from Lemmas \ref{lem:Lnu} and \ref{lem:TaR}, we obtain
\begin{enumerate}[$\bullet$]
\item if $x\leq a$,
$$\int_0^{+\infty}e^{-\lambda t}\E_x\left[F_{g_a^{(t)}}\right]dt\; \equi_{\lambda\rightarrow 0} \;\frac{1}{\lambda u_\lambda(a,a)}\E_x\left[\int_0^{+\infty} F_udL_u^a\right]
$$
since $\displaystyle \lim_{\lambda\rightarrow 0}\int_0^{+\infty} e^{-\lambda t}\Pb_x(T_a\geq t) dt=\E_x\left[T_a\right]<+\infty$,
\item if $x>a$,
$$\int_0^{+\infty}e^{-\lambda t}\E_x\left[F_{g_a^{(t)}}\right]dt \;\equi_{\lambda\rightarrow 0}\; \frac{1}{\lambda u_\lambda(a,a)}\left(\E_x[F_0](s(x)-s(a))+\E_x\left[\int_0^{+\infty} F_udL_u^a\right]\right).$$
\end{enumerate}
Therefore, for every $x\geq0$:
$$\int_0^{+\infty}e^{-\lambda t}\E_x\left[F_{g_a^{(t)}}\right]dt \;\equi_{\lambda\rightarrow 0}\;\left(\E_x[F_0](s(x)-s(a))^++\E_x\left[\int_0^{+\infty} F_udL_u^a\right]\right)  \frac{\Gamma(\beta)}{\lambda^{1-\beta}}\eta\left(\frac{1}{\lambda}\right) $$
and Point 2. follows from the Tauberian theorem (\ref{eq:Taub}) since $t\longmapsto \E_x\left[F_{g_a^{(t)}}\right]$ is decreasing.\\
\qed

\subsection{Proof of Theorem \ref{theo:equi} when $\nu\in \mathcal{L}$}

 Let $(F_t, t\geq0)$ be a positive and predictable process such that
 $$\displaystyle0<\E_x\left[\int_0^{+\infty} F_u dL_u^a\right]<\infty.$$
 From Propositions \ref{prop:ga} and \ref{prop:+-ga} we have the decomposition:
\begin{align}
\notag\int_0^t\E_x\left[F_{g_a^{(s)}}\right]ds&=\int_0^t   \int_0^s \E_x\left[F_u|g_a^{(s)}=u\right]\Pb(g_a^{(s)}\in du)\, ds \\
\label{eq:IFga}&= \E_x\left[F_0\right] \int_0^t \Pb_x(T_a>s)ds + \int_0^t  \int_0^s\E_x\left[F_u|X_u=a\right]  q(u,a,x)\nu^{(a)}([s-u,+\infty[)du\, ds.
\end{align}
But, inverting the Laplace transform (\ref{eq:LTa}), we deduce that:
$$\Pb_x(T_a>s)=\int_0^s (q(u,a,a)-q(u,a,x))\nu^{(a)}([s-u,+\infty[)du,$$
hence, we may rewrite:
\begin{equation*}
\int_0^t\E_x\left[F_{g_a^{(s)}}\right]ds=\int_0^t  f\ast \overline{\nu}^{(a)}(s)ds 
\end{equation*}
with $f(u)=\E_x[F_0]  (q(u,a,a)-q(u,a,x)) +\Pb^{x,u,a}(F_u)  q(u,x,a)$ and $\overline{\nu}^{(a)}(u)=\nu^{(a)}([u,+\infty[)$. As in the previous section, the study of the asymptotic (when $t\rightarrow +\infty$) will rely on a few Lemmas.

\subsubsection{Asymptotic of an integrated convolution product}

\begin{lemma}\label{lem:fnu}
Let $\mu$ be a measure whose tail $\overline{\mu}(t)=\mu([t,+\infty[)$ satisfies the following property:
$$\text{for every $u\geq0$}, \qquad \int_0^{t-u}\overline{\mu}(s)ds\; \mathop{\sim}_{t\rightarrow+\infty} \; \int_0^{t} \overline{\mu}(s)ds,$$
and let $f:\R^+\rightarrow\R$ be a continuous function such that $\int_0^{+\infty}f(u)du<+\infty$. Then,
$$\int_0^t f\ast \overline{\mu}(s) \; ds \mathop{\sim}_{t\rightarrow+\infty} \int_0^{+\infty}f(u)du\; \int_0^t\overline{\mu}(s)ds .   $$
\end{lemma}

\begin{proof}
Let $\varepsilon>0$. There exists $A>0$ such that, for every $t\geq A$, $\displaystyle\left|\int_t^{+\infty}f(u)du \right| <\varepsilon $.
From Fubini's Theorem, we may write:
\begin{align*}
\int_0^t f\ast \overline{\mu}(s)ds&=\int_0^t f(u)du \int_u^{t} \overline{\mu}(s-u)ds\\
&=\int_0^t f(u)du \int_0^{t-u} \overline{\mu}(s)ds\\
&=\int_0^A f(u)du \int_0^{t-u} \overline{\mu}(s)ds + \int_A^t f(u)du \int_0^{t-u} \overline{\mu}(s)ds
\end{align*}
Using this decomposition, we obtain
\begin{align}
\notag&\left|\int_0^{+\infty}f(u)du-\frac{\int_0^t f\ast \overline{\mu}(s)ds}{\int_0^{t}\overline{\mu}(s)ds}\right|\\
\notag&\qquad \leq\left| \int_0^{A}  f(u) \left(1- \frac{\int_0^{t-u}\overline{\mu}(s)ds}{\int_0^{t}\overline{\mu}(s)ds}\right)du\right| + \left|\int_A^{t}  f(u) \frac{\int_0^{t-u}\overline{\mu}(s)ds}{\int_0^{t}\overline{\mu}(s)ds}du  \right| + \left|\int_A^{+\infty}f(u)du\right|\\
\label{eq:A}&\qquad\leq \int_0^A |f(u)| \left(1- \frac{\int_0^{t-A}\overline{\mu}(s)ds}{\int_0^{t}\overline{\mu}(s)ds}\right)du + \left|\int_A^{t}  f(u) \frac{\int_0^{t-u}\overline{\mu}(s)ds}{\int_0^{t}\overline{\mu}(s)ds}du  \right| +\varepsilon.
\end{align}
Then, applying the second mean value theorem, there exists $c\in]A,t[$ such that
$$\int_A^{t}  f(u) \frac{\int_0^{t-u}\overline{\mu}(s)ds}{\int_0^{t}\overline{\mu}(s)ds}du  =
 \frac{\int_0^{t-A}\overline{\mu}(s)ds}{\int_0^{t}\overline{\mu}(s)ds}  \int_A^{c}  f(u) du $$
hence,
 $$\left|\int_A^{t}  f(u) \frac{\int_0^{t-u}\overline{\mu}(s)ds}{\int_0^{t}\overline{\mu}(s)ds}du \right| =
 \frac{\int_0^{t-A}\overline{\mu}(s)ds}{\int_0^{t}\overline{\mu}(s)ds} \left| \int_A^{+\infty}  f(u) du-\int_c^{+\infty}f(u)du\right| \leq 2\varepsilon\;  \frac{\int_0^{t-A}\overline{\mu}(s)ds}{\int_0^{t}\overline{\mu}(s)ds}$$
and, letting $t$ tend to $+\infty$ in (\ref{eq:A}), we finally obtain:
$$\lim_{t\rightarrow+\infty}\left|\int_0^{+\infty}f(u)du-\frac{\int_0^t f\ast \overline{\mu}(s)ds}{\int_0^{t}\overline{\mu}(s)ds}\right|\leq 3\varepsilon.$$
\end{proof}

\begin{remark}\label{rem:fnuL}
Assume that $\nu\in \L$. Then $\nu$  satisfies the hypothesis of Lemma \ref{lem:fnu}. Indeed
 for $u\geq0$, since $\overline{\nu}(s-u)\, \equi_{s\rightarrow+\infty}  \overline{\nu}(s)$ and $\overline{\nu}$ is not integrable at $+\infty$, we have:
$$\int_0^t \overline{\nu}(s) ds \,  \equi_{t \rightarrow+\infty}\,\int_u^t  \overline{\nu}(s) ds 
 \,\equi_{t \rightarrow+\infty}\,\int_u^t \overline{\nu}(s-u)ds=\int_0^{t-u} \overline{\nu}(s) ds. $$
\end{remark}

\begin{lemma}\label{lem:Iqq}
The following formula holds, for $x>a$:
$$
\int_0^{+\infty}  (q(u,a,a)-q(u,a,x)) du =s(x)-s(a).
$$
\end{lemma}

\begin{proof}
We set $f(t)=\int_0^t (q(u,a,a)-q(u,a,x)) du$. From Borodin-Salminen \cite[p.21]{BS}, we have:
$$f(t)=\E_{a}\left[L_t^a\right]- \E_{a}\left[L_t^x\right].$$
Since $(N_t^{(a)}=(s(X_t)-s(a))^+-L_t^a, t\geq0)$ is a martingale (see Section 6), this relation may be rewritten:
\begin{multline*}
f(t)=\E_a\left[(s(X_t)-s(a))^+\right]-\E_a\left[(s(X_t)-s(x))^+\right]\\=(s(x)-s(a))\Pb_a(X_t\geq x) + \E_a\left[(s(X_t)-s(a))1_{\{a\leq X_t \leq x\}}\right].
\end{multline*}
Then
\begin{align*}
\left|f(t)-(s(x)-s(a))\right|&\leq(s(x)-s(a))\Pb_a(X_t\leq x) +\E_a\left[(s(X_t)-s(a))1_{\{a\leq X_t \leq x\}}\right]\\
&\leq(s(x)-s(a))\left( \Pb_a(X_t\leq x) +\Pb_a(a\leq X_t\leq x) \right)\\
&\leq2(s(x)-s(a))\Pb_a(X_t\leq x)\\
&\leq 2(s(x)-s(a))\Pb_0(X_t\leq x)\xrightarrow[t\rightarrow +\infty]{}0
\end{align*}
from \cite[Chap.8, p.226]{PRY}, since $(X_t, t\geq0)$ is null recurrent.
\end{proof}

\begin{lemma}\label{lem:equiL}
Assume that $\nu$ belongs to class $\mathcal{L}$.  Then:
$$\forall a\geq0, \qquad \int_0^t\nu^{(a)}([s,+\infty[)ds\; \equi_{t\rightarrow+\infty}\;\int_0^t\nu([s,+\infty[)ds$$
\end{lemma}
\begin{proof}
Let us define the function:
$$f(t)=\int_0^{t}q(u,0,0)\nu^{(a)}([t-u, +\infty[)du.$$
We claim that $\displaystyle \lim_{t\rightarrow+\infty} f(t)=1$. Indeed, let us decompose $f$ as follows, with $\varepsilon>0$:
\begin{align*}
f_a(t)&=\int_0^{t}(q(u,0,0)-q(u,0,a))\nu^{(a)}([t-u, +\infty[)du+\Pb_0(T_a\leq t)\\
&=\int_0^{t-\varepsilon}(q(u,0,0)-q(u,0,a))\nu^{(a)}([t-u, +\infty[)du\\
&\quad\qquad +\int_{t-\varepsilon}^{t}(q(u,0,0)-q(u,0,a))\nu^{(a)}([t-u, +\infty[)du+\Pb_0(T_a\leq t).\\
&=\int_0^{+\infty}(q(u,0,0)-q(u,0,a))1_{\{u\leq t-\varepsilon\}}\nu^{(a)}([t-u, +\infty[)du\\
&\quad\qquad +\int_{0}^{\varepsilon}(q(t-u,0,0)-q(t-u,0,a))\nu^{(a)}([u, +\infty[)du+\Pb_0(T_a\leq t).
\end{align*}
From \cite[Chap.8, p.224]{PRY}, we know that for every $u\geq0$ the function $z\longmapsto q(u,0,z)$ is decreasing, hence the function
$$u\longmapsto q(u,0,0)-q(u,0,a)$$
is a positive and integrable function from Lemma \ref{lem:Iqq}. Therefore, from the dominated convergence theorem, the first integral tends toward 0 as $t\rightarrow+\infty$.  Moreover, it is known from Salminen \cite{Sal4} that for every $\displaystyle x,y\geq0,$ $$\lim_{t\rightarrow+\infty} q(t,x,y)=\frac{1}{m(\R^+)}=0,$$
which proves, still from the dominated convergence theorem, that the second integral also tends toward 0 as $t\rightarrow+\infty$. Finally,  we deduce that $\displaystyle \lim_{t\rightarrow+\infty} f_a(t)=\Pb_0(T_a<+\infty)=1$. \\
Observe now that, since $\overline{\nu}\ast q(t)=\int_0^t \nu([u, +\infty[)q(t-u,0,0)du=1$, we have from Fubini-Tonelli:
$$\int_0^t \nu^{(a)}([s,+\infty[)ds = 1 \ast \overline{\nu}^{(a)} (t) = (\overline{\nu}\ast q)\ast \overline{\nu}^{(a)}(t)= \overline{\nu} \ast f_a(t) = \int_0^t f_a(s)\nu([t-s, +\infty[)ds.$$
Let $\varepsilon>0$. There exists $A>0$ such that, for every $s\geq A$:
$$1-\varepsilon \leq f(s) \leq 1+\varepsilon.$$
Integrating this relation, we deduce that, for $t>A$:
$$(1-\varepsilon) \int_A^t \overline{\nu}(t-s)ds \leq \int_A^t  f_a(s) \overline{\nu}(t-s)ds \leq (1+\varepsilon) \int_A^t  \overline{\nu}(t-s)ds.$$
Therefore:
$$\left|\int_0^t  f_a(s) \overline{\nu}(t-s)ds -\int_A^t \overline{\nu}(t-s)ds - \int_0^A  f_a(s) \overline{\nu}(t-s)ds \right| \leq  \varepsilon \int_A^t  \overline{\nu}(t-s)ds=\varepsilon \int_0^{t-A}  \overline{\nu}(s)ds,$$
and it only remains to divide both terms by $\int_0^t \overline{\nu}(s)ds$ and let $t$ tend toward $+\infty$ to conclude, thanks to Remark \ref{rem:fnuL}, that:
$$\left|\lim_{t\rightarrow+\infty}\frac{\int_0^t \overline{\nu}^{(a)}(s)ds}{\int_0^t \overline{\nu}(s)ds} -1 \right| \leq  \varepsilon.$$
\end{proof}

\subsubsection{Proof of Point 1. of Theorem \ref{theo:equi}}
Going back to (\ref{eq:IFga}), we have, with $f(u)=\Pb^{x,u,a}(F_u)  q(u,x,a)$ and $\overline{\nu}^{(a)}(u)=\nu^{(a)}([u,+\infty[)$:
$$\int_0^t\E_x\left[F_{g_a^{(s)}}\right]ds=\left( \E_x\left[F_0\right] \int_0^t \Pb_x(T_a>s)ds+\int_0^t  f\ast \overline{\nu}^{(a)}(s)ds \right).$$
From Lemmas \ref{lem:fnu} and \ref{lem:Iqq}, we deduce that:
$$\lim_{t\rightarrow+\infty}\frac{1}{\int_0^t \overline{\nu}(s)ds} \int_0^t \Pb_x(T_a>s)ds= (s(x)-s(a))^+$$
since, for $x\leq a$, $\displaystyle \int_0^{+\infty} \Pb_x(T_a>s)ds =\E_x\left[T_a\right]<+\infty$.
Then,  Point 1. of Theorem \ref{theo:equi} follows from Lemmas \ref{lem:fnu} and \ref{lem:equiL} and the fact that:
$$\int_0^{+\infty} f(u)du=\int_0^{+\infty} \Pb^{x,u,a}(F_u)  q(u,x,a)du = \E_x\left[\int_0^{+\infty}F_udL_u^a\right]<+\infty.$$
\qed

\section{The penalization principle}

\subsection{Preliminaries: a meta-theorem and some notations}

To prove Theorem \ref{theo:penal}, we shall apply a meta-theorem, whose proof relies mainly on 
Scheff\'e's Lemma (see Meyer \cite[p.37]{Mey3}):
\begin{theorem}[\cite{RVY}]\label{theo:meta}
Let $(\Gamma_t, t\geq0)$ be a positive stochastic process satisfying for every $t>0$, $0<\E[\Gamma_t]<+\infty$. Assume that, for every $s\geq0$:
$$\lim\limits_{t\rightarrow+\infty} \frac{\E[\Gamma_t|\F_s]}{\E[\Gamma_t]}=:M_s$$
exists a.s., and that,
$$\E[M_s]=1.$$
Then, 
\begin{itemize}
\item[i)] for every $s\geq0$ and $\Lambda_s \in \F_s$:
$$\lim\limits_{t\rightarrow+\infty} \frac{\E[1_{\Lambda_s}\Gamma_t]}{\E[\Gamma_t]}=\E[M_s 1_{\Lambda_s}].$$
\item[ii)] there exists a probability measure $\Q$ on $\left(\Omega, \F_\infty \right)$ such that for every $s\geq0$:
$$\Q(\Lambda_s)=\E[M_s 1_{\Lambda_s}].$$
\end{itemize}
\end{theorem}

In the following, we shall use Biane-Yor's notations \cite{BiY}. 
We denote by $\Omega_\text{loc}$ the set of continuous functions $\omega$ taking values in $\R^+$ and defined on an interval $[0,\xi(\omega)]\subset[0,+\infty]$.
Let $\Pb$ and $\Q$ be two probability measures, such that $\Pb(\xi=+\infty)=0$. We denote by $\Pb \circ \Q$ the image measure $\Pb \otimes \Q$ by the concatenation application :
$$
\begin{array}{cccc}
\circ:& \Omega_{\text{loc}}\times  \Omega_{\text{loc}}& \longrightarrow&  \Omega_{\text{loc}}\\
&  (\omega_1, \omega_2) & \longmapsto& \omega_1 \circ \omega_2
\end{array}
$$
defined by $\xi(\omega_1 \circ \omega_2)=\xi(\omega_1)+\xi(\omega_2)$, and  
$$
(\omega_1\circ \omega_2)(t)=\left\{
\begin{array}{ll}
\omega_1(t) & \text{ si }\;0\leq t\leq \xi(\omega_1)\\
\omega_1(\xi(\omega_1))+\omega_2(t-\xi(\omega_1)) -\omega_2(0) & \text{ si }\;\xi(\omega_1)\leq t\leq \xi(\omega_1)+\xi(\omega_2).
\end{array}
\right.
$$

To simplify the notations, we define the following measure, which was first introduced by Najnudel, Roynette and Yor  \cite{NRY}:
\begin{defi}
Let $\W_x$ be the measure defined by: 
$$\W_x=\int_0^{+\infty}du\, q(u,x,a) \Pb^{x,u,a}\circ \Pb_a^{\uparrow a} + (s(x)-s(a))^+ \Pb_x^{\uparrow a}$$
$\W_x$ is a sigma-finite measure with infinite mass.
\end{defi}
This measure enjoys many remarkable properties, and was the main ingredient in the proof of the penalization results they obtained for Brownian motion.  A similar construction was made by Yano, Yano and Yor for symmetric stable L\'evy processes, see \cite{YYY}. 
\\
With this new notation, we shall now write:
\begin{align*}
\W_x(F_{g_a})&=\E_x\left[\int_0^{+\infty} F_udL_u^a\right]+\E_x^{\uparrow a}[F_0](s(x)-s(a))^{+}\\
&=\E_x\left[\int_0^{+\infty} F_udL_u^a\right]+\E_x[F_0](s(x)-s(a))^{+}.
\end{align*}

\subsection{Proof of Point $i)$ of Theorem \ref{theo:penal}}

Let  $0\leq u\leq t$. Using Biane-Yor's notation, we write:
$$(X_s, s\leq t)= (X_s, s\leq u) \circ (X_{s+u}, 0\leq s\leq t-u)$$
hence, from the Markov property, denoting $F_{g_a^{(t)}}=F(X_s, s\leq t)$:
$$\E_x[F(X_s, s\leq t)1_{\{u\leq t\}}|\F_u]=\widehat{\E}_{X_u}\left[F((X_s, s\leq u) \circ (\widehat{X}_{s}, 0< s\leq  t-u))1_{\{u\leq t\}}\right].$$
\\
Let us assume first that $\nu\in \mathcal{R}$ and that $(F_t, t\geq0)$ is decreasing. Then, from Theorem \ref{theo:equi} with $\Gamma_t=F_{g_a^{(t)}}$:
\begin{align*}
&\lim_{t\rightarrow+\infty}\frac{\widehat{\E}_{X_u}\left[F((X_s, s\leq u) \circ (\widehat{X}_{s}, 0\leq s\leq t-u))1_{\{u\leq t\}}\right]}{\nu([t,+\infty[)}\\
&= \widehat{\E}_{X_u}\left[F((X_s, s\leq u)\circ \widehat{X}_0)\right] (s(X_u)-s(a))^+ +\widehat{\E}_{X_u}\left[\int_u^{+\infty}F((X_s, s\leq u) \circ (\widehat{X}_{s}, 0\leq s\leq  v-u))d\widehat{L}_v^a\right]\\
&=F((X_s, s\leq u) (s(X_u)-s(a))^+  +  \E_{x}\left[\int_u^{+\infty}F((X_s, s\leq u) \circ (X_{s}, 0\leq s\leq  v-u))dL_v^a|\F_u \right]\\
&= F_{g_a^{(u)}}(s(X_u)-s(a))^+ + \E_x\left[\int_u^{+\infty}F_{g_a^{(v)}}dL_v^a |\F_u \right]\\
&= F_{g_a^{(u)}}(s(X_u)-s(a))^+ + \E_x\left[\int_u^{+\infty}F_vdL_v^a |\F_u \right],
\end{align*}
hence,
$$\lim_{t\rightarrow+\infty} \frac{\E_x\left[F_{g_a^{(t)}}|\F_u\right]}{\E_x\left[F_{g_a^{(t)}}\right]}=\frac{M_u(F_{g_a})}{\W_x(F_{g_a})}.$$

\noindent
On the other hand, if $\nu\in\L$ and $\Gamma_t=\int_0^t F_{g_a^{(s)}}ds$, a similar computation gives:
\begin{multline*}
\lim_{t\rightarrow+\infty}\frac{\int_0^t\widehat{\E}_{X_u}\left[F((X_s, s\leq u) \circ (\widehat{X}_{s}, 0\leq s\leq v-u))1_{\{u\leq t\}}\right]dv }{\int_0^t \nu([s,+\infty[)ds}\\=F_{g_a^{(u)}}(s(X_u)-s(a))^+ + \E_x\left[\int_u^{+\infty}F_{v}dL_v^a |\F_u \right],
\end{multline*}
and
$$\lim_{t\rightarrow+\infty} \frac{\E_x\left[\int_0^tF_{g_a^{(s)}}ds|\F_u\right]}{\E_x\left[\int_0^tF_{g_a^{(s)}}ds\right]}=\frac{M_u(F_{g_a})}{\W_x(F_{g_a})}.$$
\noindent
Therefore, to apply Theorem \ref{theo:meta}, it remains to prove that:
$$\forall t\geq0, \qquad \E_x\left[M_t(F_{g_a})\right]=\W_x(F_{g_a}).$$ We shall make a direct computation, applying Proposition \ref{prop:ga}:\\

$\bullet$ if $x> a$,
\begin{align*}
\E_x\left[M_t(F_{g_a})\right]&=\E_x\left[  F_{g_a^{(t)}}(s(X_t)-s(a))^+ + \E_x\left[\int_t^{+\infty}F_{u}dL_u^a |\F_t \right] \right]\\
&=\int_a^{+\infty} \E_x[F_0|X_t=y, T_a>t](s(y)-s(a))\Pb_x(T_a>t,X_t\in dy)\\
&\hspace{.4cm}+   \int_0^t\int_a^{+\infty}   \Pb^{x,u,a}(F_u)q(u,a,x)\Pb_a^\uparrow(X_{t-u}\in dy) du + \int_t^{+\infty}  \Pb^{x,u,a}(F_u)q(u,a,x)du\\
&=\E_x[F_0 (s(X_t)-s(a))1_{\{t<T_a\}} ]+   \int_0^{+\infty}  \Pb^{x,u,a}(F_u)q(u,a,x)du\\
&=\E_x^{\uparrow a}[F_0](s(x)-s(a))+ \int_0^{+\infty}  \Pb^{x,u,a}(F_u)q(u,a,x)du=\W_x(F_{g_a}),
\end{align*}

$\bullet$ if $x\leq a$, then, for $y>a$, $\Pb_x\left(T_a>t, X_t \in dy\right)=0$ since $X$ has continuous paths, and the same computation leads to:
$$\E_x\left[M_t(F_{g_a})\right]=\int_0^{+\infty}  \Pb^{x,u,a}(F_u)q(u,a,x)du=\W_x(F_{g_a}).$$
\noindent
Therefore, for every $x\geq0$, $\displaystyle \E_x\left[\frac{M_t(F_{g_a})}{\W_x(F_{g_a})}\right]=1$, and the proof is completed.\\
\qed

\begin{remark}
Consider the martingale $(N_t^{(a)}=(s(X_t)-s(a))^+-L_t^a, t\geq0)$. We apply the balayage formula to the semimartingale $((s(X_t)-s(a))^+, t\geq0)$:
\begin{align*}
F_{g_a^{(t)}}(s(X_t)-s(a))^+&=F_0 (s(x)-s(a))^++  \int_0^t F_{g_a^{(u)}} d(s(X_u)-s(a))^+\\
&=F_0 (s(x)-s(a))^+ + \int_0^t F_{g_a^{(u)}} dN_u^{(a)}+\int_0^t F_{g_a^{(u)}} dL_u^a\\
&=F_0 (s(x)-s(a))^+ + \int_0^t F_{g_a^{(u)}} dN_u^{(a)}+\int_0^t F_{u} dL_u^a.
\end{align*}
Therefore, the martingale $(M_t(F_{g_a}), t\geq0)$ may be rewritten:
$$M_t(F_{g_a})=F_0 (s(x)-s(a))^+ +  \int_0^t F_{g_a^{(u)}} dN_u^{(a)} + \E_x\left[\int_0^{+\infty}F_{s}dL_s^a |\F_u \right].$$
\end{remark}

\section{An integral representation of $\Q_x^{(F)}$}

Finally, Point 2. of Theorem \ref{theo:penal} is a direct consequence of the following result:
\begin{theorem}\label{theo:IQ}
$\Q_x^{(F)}$ admits the following integral representation:
$$\Q_x^{(F)} =\frac{1}{\W_x(F_{g_a})} \left(\int_0^{+\infty} q(u,x,a) F_u \Pb^{x,u,a} \circ \Pb_a^{\uparrow a} + (s(x)-s(a)) F_0\Pb_x^{\uparrow a}\right)$$
\end{theorem}

\begin{proof}
Let $G,H$ and $\varphi$ be three Borel bounded functionals. We write:
\begin{align*}
&\W_x(F_{g_a}) \Q_x^{(F)}\left(G(X_s, s\leq g_a^{(t)})\varphi(g_a^{(t)})H(X_{g_a^{(t)}+s}, s\leq t-g_a^{(t)})\right)\\
&=\E_x\left[G(X_s, s\leq g_a^{(t)})\varphi(g_a^{(t)})H(X_{g_a^{(t)}+s}, s\leq t-g_a^{(t)})M_t(F_{g_a})\right]\\ 
&=\E_x\left[G(X_s, s\leq g_a^{(t)})\varphi(g_a^{(t)})H(X_{g_a^{(t)}+s}, s\leq t-g_a^{(t)})\left(
F_{g_a^{(t)}}(s(X_t)-s(a))^+ + \E_x\left[\int_t^{+\infty}F_{u}dL_u^a |\F_t \right] 
\right)\right]\\ 
&=I_1(t) + I_2(t).
\end{align*}
On the one hand, $I_2$ equals
$$I_2(t)=\E_x\left[G(X_s, s\leq g_a^{(t)})\varphi(g_a^{(t)})H(X_{g_a^{(t)}+s}, s\leq t-g_a^{(t)})\int_t^{+\infty}F_{u}dL_u^a \right]\xrightarrow[t\rightarrow+\infty]{} 0 $$
 from the dominated convergence theorem.\\
 On the other hand, from Propositions \ref{prop:ga} and \ref{prop:+-ga}:
 \begin{align*}
 I_1(t)&= \int_a^{+\infty}\int_0^t \Pb_x\left(g_a^{(t)}\in du, X_t \in dy\right)\;\times\\
& \hspace*{1cm}\E_x\left[G(X_s, s\leq u)\varphi(u)H(X_{u+s}, s\leq t-u) F_u (s(y)-s(a))| g_a^{(t)}=u, X_t=y\right] \\
&= \int_a^{+\infty}\int_0^t \Pb_x\left(g_a^{(t)}\in du, X_t \in dy\right)\times\\
& \hspace*{1cm}\Pb^{x,u,a}\left(G(X_s, s\leq u)  F_u\right)  \varphi(u)(s(y)-s(a))\E_x\left[H(X_{u+s}, s\leq t-u)  | g_a^{(t)}=u, X_t=y\right]. 
 \end{align*}
We now separate the two cases $g_a^{(t)}=0$ and $g_a^{(t)}>0$ as in relation (\ref{eq:gatX}). \\

\noindent
$\bullet$ First, when $g_a^{(t)}=0$ and $x\leq a$, this term is null. Indeed, for $x\leq a<y$, $\Pb_x\left(T_a>t, X_t \in dy\right)=0$ since $X$ has continuous paths. Next, for $x> a$:
\begin{align*}
&\int_a^{+\infty} \Pb_x\left(T_a>t, X_t \in dy\right)G(x)\E_x[F_0]  \varphi(0)(s(y)-s(a))\E_x\left[H(X_{s}, s\leq t)  | T_a>t, X_t=y\right]\\
&=G(x)\E_x[F_0]  \varphi(0)\E_x\left[(s(X_t)-s(a))^+H(X_{s}, s\leq t)  1_{\{T_a>t\}}\right]\\
&=G(x)\E_x[F_0]  \varphi(0)(s(x)-s(a))\E_x^{\uparrow a}\left[H(X_{s}, s\leq t) \right]\\
&\xrightarrow[t\rightarrow +\infty]{} G(x)\E_x[F_0]  \varphi(0)(s(x)-s(a))^+\E_x^{\uparrow a}\left[H(X_{s}, s\geq 0) \right].
\end{align*}

\noindent
$\bullet$ Second, when $g_a^{(t)}>0$:
\begin{align*}
& \int_a^{+\infty}\int_0^t \frac{q(u,x,a)}{s(y)-s(a)}\Pb_a^\uparrow(X_{t-u}\in dy) du\; \times\\
& \hspace*{1cm}\Pb^{x,u,a}\left(G(X_s, s\leq u)  F_u\right)  \varphi(u) (s(y)-s(a))\E_x\left[H(X_{u+s}, s\leq t-u) | g_a^{(t)}=u, X_t=y\right]\\
&= \int_a^{+\infty}\int_0^t q(u,x,a)\Pb_a^\uparrow(X_{t-u}\in dy) du\; \times\\
& \hspace*{1cm}\Pb^{x,u,a}\left(G(X_s, s\leq u)  F_u\right)  \varphi(u)\E_a^{\uparrow a}\left[H(X_{s}, s\leq t-u) | X_{t-u}=y\right]\\
&= \int_0^t du\;q(u,x,a) \Pb^{x,u,a}\left(G(X_s, s\leq u)  F_u\right)  \varphi(u)\E_a^{\uparrow a}\left[H(X_{s}, s\leq t-u)\right]\\ 
&\xrightarrow[t\rightarrow +\infty]{}\int_0^{+\infty } du\;q(u,x,a) \Pb^{x,u,a}\left(G(X_s, s\leq u)  F_u\right)  \varphi(u)\E_a^{\uparrow a}\left[H(X_{s}, s\geq 0)\right].
\end{align*}
\end{proof}

\begin{remark}
From Theorem \ref{theo:IQ}, $\Q_x^{(F)}(g_a<+\infty)=1$ and we deduce that, conditionally to $g_a$, 
\begin{enumerate}
\item on the event $g_a>0$, the law of the process $(X_{g_a+u}, u\geq0)$ under $\Q_x^{(F)}$ is the same as the law of $(X_u, u\geq0)$ under $\Pb_a^{\uparrow a}$,
\item on the event $g_a=0$, the law of the process $(X_{u}, u\geq0)$ under $\Q_x^{(F)}$ is the same as the law of $(X_u, u\geq0)$ under $\Pb_x^{\uparrow a}$.
\end{enumerate}
Observe that the process $(F_u, u\geq0)$ plays no role in these results. 
\end{remark}

\begin{example}
Let $h$ be a positive and decreasing function on $\R^+$. 
\begin{enumerate}[$\bullet$]
\item Let us take $(F_t, t\geq0)=(h(L_t^a), t\geq0)$ and assume that $\displaystyle\int_0^{+\infty}h(\ell)d\ell=1$:
$$ \Q_0^{(h(L_{g_a}^a))}=\int_0^{+\infty} du\,q(u,0,a) h(L_u^a) \Pb^{0,u,a}\circ \Pb_a^\uparrow.$$
Thus, under $\Q_0^{(h(L_{g_a}^a))}$, the r.v. $L_\infty^a$ is a.s. finite and admits $\ell\longmapsto h(\ell)$ as its density function. Furthermore,  
conditionally to $L_\infty^a=\ell$ the process $(X_t, t\leq g_a)$ has the same law as $(X_t, t\leq \tau^{(a)}_\ell)$ under $\Pb_0$.

\item Let us take $(F_t, t\geq0) = (h(t), t\geq0)$ and  assume that $\displaystyle\int_0^{+\infty}h(u)q(u,0,a)du=1$:
$$ \Q_0^{(h(g_a))}=\int_0^{+\infty} du\,q(u,0,a) h(u) \Pb^{0,u,a}\circ \Pb_a^\uparrow.$$
Then, under $\Pb_0^{(h(g_a))}$, the r.v. $g_a$ admits as density function $u\longmapsto h(u)q(u,0,a)$ and, conditionally to $g_a=u$ the process $(X_t, t\leq g_a)$ has the same law as $(X_t, t\leq u)$ under $\Pb^{0,u,a}$.
\end{enumerate}
\end{example}

\section{Appendix}
Let $a\geq 0$ and define  $(N_t^{(a)}:= (s(X_t)-s(a))^+-L_t^a, t\geq0)$. The aim of this section is to prove that $(N_t^{(a)}, t\geq0)$ is a martingale in the filtration $(\F_t, t\geq0)$.
Applying the Markov property to the diffusion $(X_t, t\geq0)$ we deduce that:
$$\E_0\left[N_{t+s}^{(a)}|\F_s \right]=\E_{X_s}\left[(s(X_t)-s(a))^+\right]-L_s^a -\E_{X_s}\left[L_t^a\right].$$
We set $x=X_s$, so we need to prove that for every $x\geq0$:
\begin{equation*}
(s(x)-s(a))^+=\E_x\left[(s(X_t)-s(a))^+\right]-\E_x\left[L_t^a\right],
\end{equation*}
or rather:
$$\int_0^{+\infty} (s(y)-s(a))^+ q(t,x,y)m(dy)= \int_0^t q(u,x,a)du + (s(x)-s(a))^+.$$
Let us take the Laplace transform of this last relation (applying Fubini-Tonelli):
\begin{equation}\label{eq:LNa}
\int_0^{+\infty} (s(y)-s(a))^+ u_\lambda(x,y)m(dy)= \frac{u_\lambda(x,a)}{\lambda} + \frac{(s(x)-s(a))^+ }{\lambda}.
\end{equation}
Our aim now is to prove (\ref{eq:LNa}). To this end, 
we shall use the following representation of the resolvent kernel $u_\lambda(x,y)$ (see \cite[p.19]{BS}):
$$u_\lambda(x,y)=\omega_\lambda^{-1} \psi_\lambda(x)\varphi_\lambda(y)\qquad \qquad x\leq y$$
where $\psi_\lambda$ and $\varphi_\lambda$ are the fundamental solutions of the generalized differential equation
\begin{equation}\label{eq:Gen}
\frac{d^2 }{d m\; d s} u =\lambda u
\end{equation}
such that $\psi_\lambda$ is increasing (resp.  $\varphi_\lambda$ is decreasing) and the Wronskian $\omega_\lambda$ is given, for all $z\geq0$ by:
$$\omega_\lambda = \varphi_\lambda(z) \frac{ d\psi_\lambda}{ds}(z)- \psi_\lambda(z) \frac{ d\varphi_\lambda}{ds}(z).$$
Note that since $m$ has no atoms, the meaning of (\ref{eq:Gen}) is as follows:
$$\forall y\geq x, \quad \lambda \int_x^y u(z) m(dz) = \frac{d\,u}{ds}(y)-\frac{d\,u}{ds}(x) \qquad \text{where} \quad \frac{d\,u}{ds}(x):=\lim_{h\rightarrow0}\frac{u(x+h)-u(x)}{s(x+h)-s(x)}.$$

\noindent
$\bullet$ Assume first that $x\leq a$.
\begin{align*}
&\int_a^{+\infty} (s(y)-s(a)) u_\lambda(x,y)m(dy)\\
&=\frac{1}{\omega_\lambda}  \int_a^{+\infty} \left( \int_a^y ds(z) \right)  \psi_\lambda(x)\varphi_\lambda(y)m(dy)\\
&= \frac{ \psi_\lambda(x)}{\omega_\lambda} \int_a^{+\infty}  ds(z)  \int_z^{+\infty}\varphi_\lambda(y) m(dy)  \quad \text{(applying Fubini-Tonelli's theorem since $\varphi_\lambda\geq0$)}\\
&=-\frac{\psi_\lambda(x)}{\lambda\omega_\lambda} \int_a^{+\infty} ds(z)     \frac{d \varphi_\lambda}{ds}(z) \quad \left(\text{since }\lim_{y\rightarrow+\infty} \frac{d\varphi_\lambda}{ds}(y)=0\quad \text{as $+\infty$ is a natural boundary}\right) \\
&= \frac{\psi_\lambda(x)}{\lambda\omega_\lambda}\; \varphi_\lambda(a)\quad \left(\text{since }\lim_{z\rightarrow+\infty} \varphi_\lambda(z)=0\quad \text{as $+\infty$ is a natural boundary}\right)\\ 
&=\frac{u_\lambda(x,a)}{\lambda}
\end{align*}
which gives (\ref{eq:LNa}) for $x\leq a$.\\

\noindent
$\bullet$ Now, let us suppose that $x>a$. We have, with the same computation:
\begin{align*}
&\int_a^{+\infty} (s(y)-s(a)) u_\lambda(x,y)m(dy)\\
&=\int_a^{x} (s(y)-s(a)) u_\lambda(x,y)m(dy)+\int_x^{+\infty} (s(y)-s(a)) u_\lambda(x,y)m(dy)\\
&=I_1+I_2.
\end{align*}
On the one hand:
\begin{align*}
I_1&=\frac{\varphi_\lambda(x)}{\omega_\lambda}\int_a^{x} ds(z) \int_z^{x} \psi_\lambda(y)m(dy)\\
&= \frac{\varphi_\lambda(x)}{\lambda \omega_\lambda} \int_a^{x} ds(z) \left(\frac{d \psi_\lambda}{ds}(x)-\frac{d \psi_\lambda}{ds}(z)\right)\\
&= \frac{ \varphi_\lambda(x)}{\lambda\omega_\lambda}\left((s(x)-s(a)) \frac{d \psi_\lambda}{ds}(x) -  \left(\psi_\lambda(x)-\psi_\lambda(a)\right)\right)\\
&= \frac{s(x)-s(a)}{\lambda \omega_\lambda}\varphi_\lambda(x)\frac{d \psi_\lambda}{ds}(x) 
- \frac{u_\lambda(x,x)}{\lambda} + \frac{u_\lambda(x,a)}{\lambda}.
\end{align*}
On the other hand:
\begin{align*}
I_2&=\int_x^{+\infty} (s(y)-s(x)) u_\lambda(x,y)m(dy)+ (s(x)-s(a)) \int_x^{+\infty}u_\lambda(x,y)m(dy)\\
&= \frac{u_\lambda(x,x)}{\lambda} +\frac{s(x)-s(a)}{\omega_\lambda} \psi_\lambda(x) \int_x^{+\infty}\varphi_\lambda(y)m(dy)\quad \text{(from the previous computations)}\\
&=\frac{u_\lambda(x,x)}{\lambda} - \frac{s(x)-s(a)}{\lambda\omega_\lambda} \psi_\lambda(x) \frac{d \varphi_\lambda}{ds}(x). 
\end{align*}
Finally, gathering both terms, we obtain for $x>a$:
\begin{align*}
\int_a^{+\infty} (s(y)-s(a)) u_\lambda(x,y)m(dy)&= \frac{s(x)-s(a)}{\lambda\omega_\lambda}\left(\varphi_\lambda(x)\frac{d \psi_\lambda}{ds}(x)-\psi_\lambda(x)\frac{d \varphi_\lambda}{ds}(x)  \right)+ \frac{u_\lambda(x,a)}{\lambda},\\
&= \frac{s(x)-s(a)}{\lambda}+ \frac{u_\lambda(x,a)}{\lambda},
\end{align*}
which is the desired result (\ref{eq:LNa}) from the definition of the Wronskian.\\
\qed

\bibliographystyle{alpha}

\end{document}